\documentclass{amsart} 

\usepackage[english]{babel}
\usepackage[latin1]{inputenc} 
\usepackage{amsmath}
\usepackage{amsthm}
\usepackage{amssymb}
\usepackage{tikz}

\usepackage{imakeidx}
\usepackage{appendix}
\usepackage{tikz-cd}
\usepackage{verbatim}
\usepackage{enumitem}
\usepackage{multicol}

\usepackage[backref=page]{hyperref}
\usepackage{cleveref}

\usepackage{mathtools}

\numberwithin{equation}{section} 
\newtheorem{thm}[equation]{Theorem} 

\newtheorem{prop}[equation]{Proposition}
\newtheorem{lemma}[equation]{Lemma}
\newtheorem{cor}[equation]{Corollary}

\theoremstyle{definition}
\newtheorem{defin}[equation]{Definition}
\newtheorem{example}[equation]{Example}
\theoremstyle{remark}
\newtheorem{rmk}[equation]{Remark}

\newcommand{\Q}{\mathbb Q}
\newcommand{\F}{\mathbb F}
\newcommand{\Z}{\mathbb Z}
\renewcommand{\L}{\mathbb L}
\newcommand{\G}{\mathbb G}
\renewcommand{\P}{\mathbb P}
\newcommand{\C}{\mathbb C}
\renewcommand{\c}{\subseteq}
\newcommand{\A}{\mathbb A}

\newcommand{\mc}[1]{\mathcal{#1}}
\newcommand{\cl}{\overline}
\newcommand{\set}[1]{\{#1\}}
\newcommand{\on}[1]{\operatorname{#1}}
\newcommand{\ang}[1]{\left \langle{#1}\right \rangle}

\title[Mixed Tate property and motivic classes]{On the Mixed Tate property and the motivic class of the classifying stack of a finite group}
\author{Federico Scavia}

\begin{document}
	\maketitle

\begin{abstract}
	Let $G$ be a finite group, and let $\set{B_{\C}G}$ the class of its classifying stack $B_{\C}G$ in Ekedahl's Grothendieck ring of algebraic $\C$-stacks $K_0(\on{Stacks}_{\C})$. We show that if $B_{\C}G$ has the mixed Tate property, the invariants $H^i(\set{B_{\C}G})$ defined by T. Ekedahl are zero for all $i\neq 0$. We also extend Ekedahl's construction of these invariants to fields of positive characteristic. 
\end{abstract}

\section{Introduction}	
Let $k$ be a field, and denote by $K_0(\on{Stacks}_k)$ the Grothendieck ring of algebraic $k$-stacks, introduced by T. Ekedahl in \cite{ekedahl2009grothendieck}. It is a commutative ring with identity $1=\set{\on{Spec}k}$. By definition, every algebraic stack $\mc{X}$ of finite type over $k$ and with affine stabilizers has a class $\set{\mc{X}}$ in $K_0(\on{Stacks}_k)$. 

Assume now that $k$ is algebraically closed of characteristic zero, and let $G$ be a finite group. In \cite[Section 9]{totaro2016motive}, B. Totaro considered the following properties of the classifying stack $B_kG$: stable rationality of $B_kG$, triviality of the birational motive of $B_kG$, the equality $\set{B_kG}=1$ in $K_0(\on{Stacks}_k)$, the weak Chow-K\"unneth property for $B_kG$, the Chow-K\"unneth property for $B_kG$, and the mixed Tate property for $B_kG$. 

As Totaro observed, all known examples suggest that these properties are equivalent. Let $V$ be a faithful $G$-representation over $k$. Up to stably birational equivalence, the quotient variety $V/G$ is independent of the choice of $V$, hence the groups $H^i_{\on{nr}}(k(V)^G,\Q/\Z(j))$ are independent of $V$. When the unramified Brauer group $\on{Br}_{\on{nr}}(k(V)^G)=H^2_{\on{nr}}(k(V)^G,\Q/\Z(1))$ is non-zero, the six properties considered by Totaro fail. A proof of the equivalence of the properties of $B_kG$ listed by Totaro seems to be out of reach of current techniques.

Following Totaro, we say that $B_kG$ has the mixed Tate property if the Voevodsky motive with compact support of $B_kG$ is a mixed Tate motive; see \Cref{triang-motives}. In this paper, we investigate the relation between the equality $\set{B_kG}=1$ in $K_0(\on{Stacks}_k)$ and the mixed Tate property for $B_kG$. There is no known implication among these two properties, in either direction. 

For every commutative ring with identity $R$, let $K_0(\on{Mod}_R)$ denote the Grothen\-dieck group of finitely generated $R$-modules, that is,  the abelian group generated by isomorphism classes $[M]$ of finitely generated $R$-modules $M$, subject to the relations $[M\oplus N]=[M]+[N]$ for every two abelian groups $M,N$.  When $k=\C$, for every integer $i$ Ekedahl constructed a group homomorphism \[H^i:K_0(\on{Stacks}_{\C})\to K_0(\on{Mod}_{\Z})\] such that $H^i(\set{X}/\L^m)=H^{i+2m}(X(\C),\Z)$ for every smooth smooth $\C$-variety $X$ and every $m\in\Z$ (note that $\L$ is invertible in $K_0(\on{Stacks}_k)$; see \Cref{prelim}). When $k$ is an algebraically closed field of characteristic zero, he gave a similar definition using $\ell$-adic cohomology. For every prime $\ell$ and every $i$, one has a group homomorphism \[H^i_{\ell}:K_0(\on{Stacks}_{k})\to K_0(\on{Mod}_{\Z_{\ell}})\] such that $H^i_{\ell}(\set{X}/\L^m)=H_{\text{\'et}}^{i+2m}(X,\Z_{\ell})$ for every smooth proper $k$-variety $X$ and every $m\in\Z$. If $k=\C$, the knowledge of $H^i(\set{B_{\C}G})$ amounts to the knowledge of $H^i_{\ell}(\set{B_{\C}G})$ for every prime $\ell$; see \Cref{comparison}. The invariants $H^i(\set{B_{\C}G})$ have been further investigated in \cite{martino2016ekedahl} and \cite{martino2017introduction}.

Ekedahl's construction makes essential use of Bittner's presentation of the Gro\-thendieck ring of varieties $K_0(\on{Var}_k)$ proved in \cite{bittner2004universal}, which depends on Hironaka's resolution of singularities. In particular, at the present state of knowledge about resolution of singularities, Ekedahl's construction does not extend to fields of positive characteristic.

We now state our main result, which gives some evidence for the equivalence between the mixed Tate property for $B_kG$ and the equality $\set{B_kG}=1$ in $K_0(\on{Stacks}_k)$.
\begin{thm}\label{ekedahl-inv}
	Let $k$ be an algebraically closed field of characteristic zero, and let $G$ be a finite group.
	\begin{enumerate}[label=(\alph*)]
		\item If $\ell$ is a prime number and $B_kG$ is $\Z_{\ell}$-mixed Tate, then $H^0_{\ell}(\set{B_kG})=[\Z_\ell]$ and $H^i_{\ell}(\set{B_kG})=0$ for every $i\neq 0$.
		\item If $k=\C$ and $B_{\C}G$ is mixed Tate, then $H^0(\set{B_{\C}G})=[\Z]$ and $H^i(\set{B_{\C}G})=0$ for every $i\neq 0$.
	\end{enumerate}
\end{thm}

As an application of \Cref{ekedahl-inv}, we prove a conjecture of I. Martino. Let $k=\C$, and for every prime $p\geq 3$ let $G_p$ be the Heisenberg group of order $p^3$. The group $G_3$ admits a faithful $3$-dimensional representation, hence $\set{B_{\C}G_3}=1$ in $K_0(\on{Stacks}_{\C})$ by \cite[Theorem 2.4]{martino2016ekedahl}. In \cite[Theorem 4.4]{martino2016ekedahl}, Martino proved that $H^i(\set{B_{\C}G_5})=0$ for every $i\neq 0$, with the help of a computer calculation. He conjectured that the same should be true for $G_p$, where $p$ is an arbitrary odd prime number; see \cite[Conjecture, p. 1296]{martino2016ekedahl}. 

\begin{cor}\label{martino-cor}
	Let $p$ be a prime number, and let $G$ be a group of order $p^3$. Then $H^i(\set{B_{\C}G})=0$ for every $i\neq 0$.
\end{cor}

Indeed, by a result of T. P{\u{a}}durariu \cite[Theorem 3.1]{padurariu2015groups}, if $G$ is a group of order $p^3$, then $B_{\C}G$ is mixed Tate. Now \Cref{martino-cor} follows from \Cref{ekedahl-inv}.

The main ingredient in our proof of \Cref{ekedahl-inv} is a new construction of the invariants $H^i_{\ell}(-)$. This alternative definition does not rely on resolution of singularities, but uses Voevodsky's triangulated category of motives, and a certain weight structure defined on it by M. Bondarko. Due to recent advances of Bondarko in the theory of Voevodsky motives, made possible by O. Gabber's prime-to-$\ell$ resolution of singularities \cite[Th\'eor\`eme 3(1) p. vi]{illusie2014travaux}, this definition also works when $k$ is of positive characteristic $p$, as long as $\ell\neq p$. 

Every element $\alpha\in K_0(\on{Mod}_{\Z_{\ell}})$ can be written in a unique way as $\alpha=n[\Z_{\ell}]+\beta$, where $n\in\Z$ and $\beta$ is a linear combination of classes of finite cyclic $\ell$-groups (a similar definition applies to elements of $K_0(\on{Mod}_{\Z})$. We say that $\alpha$ is torsion if $n=0$. In characteristic zero, Ekedahl showed that $H^{i}(\set{B_{\C}G})$ is torsion for every $i\neq 0$, $H^{i}(\set{B_{\C}G})=0$ for $i>0$ or $i=-1$, $H^0(\set{B_{\C}G})=[\Z]$, and $H^{-2}(\set{B_{\C}G})=[\on{Hom}(\on{Br}_{\on{nr}}(\C(V^G)),\Q/\Z)]$, where $V$ is a faithful complex representation of $G$; see \cite[Theorem 5.1]{ekedahl2009geometric}. We prove an analog of Ekedahl's Theorem in positive characteristic.

\begin{thm}\label{calculation}
	Let $G$ be a finite group. There exists a positive integer $N$ such that for every prime $p\geq N$, for every algebraically closed field $k$ of characteristic $p$, and every prime $\ell\neq p$, $H^i_{\ell}(\set{B_kG})$ is torsion for every $i\neq 0$ and: 
	\[
	H^i_{\ell}(\set{B_kG})=
	\begin{cases}
	0, &i>0, \\
	[\Z_{\ell}], &i=0,\\
	0, & i=-1,\\
	[\on{Hom}(\on{Br}_{\on{nr}}({k}(V)^G)\set{\ell},\Q_{\ell}/\Z_{\ell})], &i=-2.
	\end{cases}
	\]
	Here $V$ is a faithful $G$-representation over $k$.
\end{thm}
As a corollary, we give the first examples of finite groups $G$ such that $\set{B_kG}\neq 1$ in $K_0(\on{Stacks}_k)$ in positive characteristic; see \Cref{saltman}.

The proof of \Cref{calculation} consists of a spreading-out argument. The hypothesis $p\geq N$ is a technical condition that seems necessary to make our proof work, and could be removed if we knew resolution of singularities in positive characteristic.

We give a brief summary of the content of each section of the paper. In \Cref{prelim} we discuss preliminaries about Grothendieck rings, and in \Cref{weights} we cover the preliminaries on weight structures that will be needed for the proof of \Cref{ekedahl-inv}. In \Cref{sec4} we give an alternative construction of Ekedahl's $H^i_{\ell}$, and in \Cref{simplified} we give a simplified variant of this constructions, together with an application to the calculation of of the motivic class of a linear algebraic group. \Cref{ekedahl-inv} is proved in \Cref{sec-ekedahl-inv}, and \Cref{calculation} in \Cref{sec-calculation}.

\subsection*{Notation}
If $k$ is a field, we denote by $\cl{k}$ an algebraic closure of $k$. The exponential characteristic of $k$ is $1$ if $\on{char}k=0$ and is $p$ if $\on{char}k=p>0$. A $k$-variety is a separated geometrically integral $k$-scheme of finite type. If $X$ is a $k$-variety, we set $\cl{X}:=X\times_k\cl{k}$.
If $\ell$ is a prime different from $\on{char}k$, we denote by $\Z_{\ell}$ the ring of $\ell$-adic integers, and by $H^i_{\text{\'et}}(X,\Z_\ell)$ the $i$-th group of $\ell$-adic cohomology on $X$. If $k=\C$, we write $H^i(X(\C),\Z)$ for the $i$-th singular cohomology group of $X(\C)$.

If $S$ is a commutative ring with identity, and $a\in S$, we define \[\Phi_a:=\set{a, a^n-1: n\geq 1}\c S,\] and we write $\Phi_a^{-1}S$ for the localization of $S$ at the multiplicative subset generated by $\Phi_a$. If  $a,a_1,\dots,a_n\in S$, we say that $a$ is a polynomial in $a_1,\dots,a_n$ if $a$ belongs to the image of the homomorphism $\Z[x_1,\dots,x_n]\to S$ given by $x_i\mapsto a_i$. 
	
\section{Preliminaries}\label{prelim}

\subsection{Grothendieck rings of varieties and motives}
Let $k$ be an arbitrary field. By definition, the Grothendieck ring of varieties $K_0(\on{Var}_{k})$ is the abelian group generated by isomorphism classes $\set{X}$ of $k$-schemes $X$ of finite type, modulo the relations $\set{X}=\set{Y}+\set{X\setminus Y}$ for every closed subscheme $Y\c X$. The product operation in $K_0(\on{Var}_k)$ is defined on generators by $\set{X}\cdot\set{Y}:=\set{X\times_k Y}$, and we have $1=\set{\on{Spec}k}$. We set $\L:=\set{\A_k^1}$. If $k$ is of characteristic zero, then $K_0(\on{Var}_k)$ admits the following presentation, due to F. Bittner.

\begin{thm}[Bittner]\label{bittner}
	Let $k$ be a field of characteristic zero. As an abelian group, $K_0(\on{Var}_k)$ is canonically isomorphic to the abelian group generated by isomorphism classes $\set{X}$ of smooth projective $k$-varieties $X$, modulo the relations $\set{\emptyset}=0$ and $\set{X}-\set{Y}=\set{\widetilde{X}}-\set{E}$, where $X$ is a smooth projective $k$-variety, $\widetilde{X}\to X$ is the blow-up at the smooth closed subvariety $Y\c E$, and $E$ is the exceptional divisor of the blow-up.
\end{thm}

\begin{proof}
	See \cite[Theorem 3.1]{bittner2004universal}. 
\end{proof}

Following Ekedahl \cite{ekedahl2009grothendieck}, we define the Grothendieck ring of stacks $K_0(\on{Stacks}_k)$ as the abelian group generated by isomorphism classes $\set{\mc{X}}$ of algebraic stacks $\mc{X}$ with affine stabilizers and of finite type over $k$, modulo the relations $\set{\mc{X}}=\set{\mc{Y}}+\set{\mc{X}\setminus \mc{Y}}$ for every closed embedding $\mc{Y}\c \mc{X}$, and the relations $\set{\mc{E}}=\set{\A_k^r\times_k \mc{X}}$ for every vector bundle $\mc{E}\to\mc{X}$ of constant rank $r$. The product is defined on generators by $\set{\mc{X}}\cdot\set{\mc{Y}}:=\set{\mc{X}\times_k\mc{Y}}$, and we have $1=\set{\on{Spec}k}$. By \cite[Theorem 1.2]{ekedahl2009grothendieck}, the canonical ring homomorphism $K_0(\on{Var}_k)\to K_0(\on{Stacks}_k)$ induces an isomorphism
\[K_0(\on{Stacks}_k)\cong \Phi_{\L}^{-1}K_0(\on{Var}_k).\]

\begin{rmk}
Several definitions of a Grothendieck ring of algebraic stacks predate that of Ekedahl; see \cite{behrend2007motivic}, \cite{joyce2007motivic} and \cite{toen2005grothendieck}. 
\end{rmk}

There is a filtration $\on{Fil}^{\bullet}K_0(\on{Var}_k)$, such that for every integer $n$ the subgroup $\on{Fil}^nK_0(\on{Var}_k)$ is generated by elements of the form $\set{X}/\L^m$, where $X$ is a $k$-variety and $\on{dim}(X)-m\leq n$. We denote by $\hat{K}_0(\on{Var}_k)$ the completion of $K_0(\on{Var}_k)$ with respect to this filtration. Since the filtration is compatible with multiplication, $\hat{K}_0(\on{Var}_k)$ inherits a ring structure. For every $n\geq 1$, we have $(1-\L^n)\sum_{i\geq 0} \L^{ni}=1$ in $\hat{K}_0(\on{Var}_k)$. Therefore, we have canonical ring homomorphisms
\[K_0(\on{Var}_k)\to K_0(\on{Stacks}_k)\to \hat{K}_0(\on{Var}_k).\]

During the proof of \Cref{calculation}, we will use a generalization of $K_0(\on{Var}_k)$. If $S$ is a scheme, we define $K_0(\on{Sch}_S)$ as the abelian group generated by isomorphism classes $\set{X}$ of $S$-schemes $X$, modulo the relations $\set{X}=\set{Y}+\set{X\setminus Y}$ for every closed subscheme $Y\c X$. The fibered product over $S$ makes $K_0(\on{Sch}_S)$ into a commutative ring with identity. By definition, $K_0(\on{Var}_k)= K_0(\on{Sch}_k)$. For every morphism $f:T\to S$, we have an induced ring homomorphism $f^*:K_0(\on{Sch}_S)\to K_0(\on{Sch}_T)$. 

\begin{rmk}\label{smoothsuffice}
	Assume that $k$ is a perfect field. Denote by $K'_0(\on{Var}_k)$ the abelian group generated by classes $\set{X}$ of smooth $k$-schemes of finite type, modulo the relations $\set{X}=\set{Y}+\set{X\setminus Y}$ for every closed embedding of smooth $k$-schemes $Y\hookrightarrow X$. The fibered product over $k$ makes $K'_0(\on{Var}_k)$ into a commutative ring with identity. There is a canonical ring homomorphism $K'_0(\on{Var}_k)\to K_0(\on{Var}_k)$. Using the defining presentations of $K'_0(\on{Var}_k)$ and $K_0(\on{Var}_k)$, it is easy to show that this homomorphism is injective. It is also surjective, because every reduced $k$-scheme admits a locally-closed stratification with finitely many strata and such that the every stratum is smooth over $k$; here we use that $k$ is perfect. 
\end{rmk}

\subsection{Grothendieck rings of additive categories}
We recall some standard definitions; see \cite[3.2.1]{gillet1996descent} or \cite[5.3]{bondarko2010weight}. 

If $\mc{A}$ is an additive category, we define the Grothendieck group of $\mc{A}$ as the abelian group $K_0(\mc{A})$ generated by isomorphism classes $[A]$ of objects $A$ of $\mc{A}$ modulo the relations $[A\oplus B]=[A]+[B]$ for every $A,B\in \mc{A}$. If $\mc{A}$ is a monoidal additive category, the tensor product on $\mc{A}$ induces a multiplication on $K_0(\mc{A})$, which makes $K_0(\mc{A})$ into a commutative ring with identity.

\begin{example}
	Let $R$ be a commutative ring with identity, and let $\on{Mod}_R$ denote the category of finitely generated $R$-modules. If $R$ is a principal ideal domain then, as an abelian group, $K_0(\on{Mod}_R)$ is freely generated by the classes of $R$ and $R/P^n$, where $P\c R$ is a non-zero prime ideal and $n\geq 1$; see \cite[Proposition 3.3(i)]{ekedahl2009grothendieck}.
\end{example}
Let $\mc{C}$ be a triangulated category. Recall that this means that $\mc{C}$ is an additive category together with a translation (or suspension) functor $X\mapsto X[1]$, and a class of distinguished triangles; see \cite[IV.1]{gelfand2013methods}. We define the Grothendieck group $K_0^{\on{tr}}(\mc{C})$ of $\mc{C}$ as the abelian group generated by isomorphism classes $[X]$ of objects $X\in \mc{C}$, modulo the relations $[X]=[Y]+[Z]$ for every distinguished triangle \[Y\to X\to Z\to Y[1].\] Of course, $K_0^{\on{tr}}(\mc{C})$ is in general different from $K_0(\mc{C})$. 

If $\mc{C}$ is a tensor triangulated category, $K_0^{\on{tr}}(\mc{C})$ has a natural ring structure.

\subsection{The derived category of representations of a profinite group}
Let $\mc{G}$ be a profinite group, and let $R$ be a topological commutative ring with identity. If $\ell$ is a prime number, we will use $R=\Z,\Z/\ell$ with the discrete topology, or $\Z_{\ell},\Q_\ell$ with the $\ell$-adic topology. We denote by $\on{Rep}_{\mc{G},R}$ the category whose objects are finitely generated continuous $R$-modules, equipped with a continuous $R$-linear $\mc{G}$-action. 

We denote by $D^b(\on{Rep}_{\mc{G},R})$ the bounded derived category of $\on{Rep}_{\mc{G},R}$. We define \[L_0(\on{Rep}_{\mc{G},R}):=K_0(D^b(\on{Rep}_{\mc{G},R})).\] This is the Grothendieck ring of $D^b(\on{Rep}_{\mc{G},R})$, viewed as an additive category. The derived tensor product in $D^b(\on{Rep}_{\mc{G},R})$ endows $L_0(\on{Rep}_{\mc{G},R})$ with the structure of a commutative ring with identity: $[M]\cdot [N]:=[M\otimes_R^{L}N]$.

There is a degree filtration $\on{Fil}^{\bullet}L_0(\on{Rep}_{\mc{G},R})$ on $L_0(\on{Rep}_{\mc{G},R})$, where for every integer $n$, the subgroup $\on{Fil}^{n}L_0(\on{Rep}_{\mc{G},R})$ consists of complexes concentrated in degrees $\leq n$. We let $\hat{L}_0(\on{Rep}_{\mc{G},R})$ be the completion of $L_0(\on{Rep}_{\mc{G},R})$ with respect to this filtration. Since the filtration is stable under multiplication, there is a canonical ring homomorphism \[L_0(\on{Rep}_{\mc{G},R})\to \hat{L}_0(\on{Rep}_{\mc{G},R}).\]

For every integer $i$, we have a group homomorphism \[h^i_R:L_0(\on{Rep}_{\mc{G},R})\to K_0(\on{Rep}_{\mc{G},R})\] given by $[(M,d)]\mapsto [\on{Ker}d^i/\on{Im}d^{i+1}]$.


\section{Preliminaries on weight structures}\label{weights}
In the proof of \Cref{ekedahl-inv}, we will make use of the theory of weight structures on triangulated categories. In this section, we provide the basic definitions and references that are needed to follow our argument. The material of this section will only be used in \Cref{sec-ekedahl-inv}.

\subsection{Idempotents and retracts} Let $\mc{A}$ be an additive category. Then $\mc{A}$ embeds into its \emph{idempotent completion} (or Karoubi envelope) $\on{Kar}(\mc{A})$. By definition, $\on{Kar}(\mc{A})$ is the universal enlargement of $\mc{A}$ in which every idempotent splits. A concrete description of $\on{Kar}(\mc{A})$ is as follows: the objects of $\on{Kar}(\mc{A})$ are pairs $(A,p)$, where $A$ is an object of $\mc{A}$, and $p:A\to A$ is a projector, that is, $p^2=p$. Homomorphisms $(A,p)\to (B,q)$ are given by homomorphisms $\varphi:A\to B$ in $\mc{A}$ such that $\varphi\circ p=\varphi=q\circ \varphi$. We say that $\mc{A}$ is \emph{idempotent complete} if every idempotent endomorphism $p:A\to A$ in $\mc{A}$ splits, or equivalently if the canonical functor $\mc{A}\to \on{Kar}(\mc{A})$ is an equivalence. 

Recall that if $\mc{A}$ is abelian or triangulated, every retract is a direct summand; see \cite[Lemma 2.1.5]{neeman2001triangulated} for the triangulated case. Let $\mc{B}$ be an additive subcategory of the additive category $\mc{A}$. We say that $\mc{B}$ is \emph{retraction-closed} (or Karoubi-closed) in $\mc{A}$ if $\mc{B}$ is closed under retracts in $\mc{A}$: if $X\in \mc{A}$ and $\on{id}_X$ factors through $Y\in \mc{B}$, then $X\in \mc{B}$. For every additive subcategory $\mc{B}$ of $\mc{A}$, we define $\on{Kar}_{\mc{A}}(\mc{B})$ as the smallest retraction-closed subcategory of $\mc{A}$ containing $\mc{B}$.

\subsection{Triangulated categories} 
Let $\mc{C}$ be a triangulated category, and let $\mc{B}$ be a full additive subcategory of $\mc{C}$. Recall that $\mc{B}$ is \emph{strictly full} if it is full and closed under isomorphisms in $\mc{C}$. We say that $\mc{B}$ is \emph{thick} if it is closed under direct summands; see \cite[Definition 2.1.6]{neeman2001triangulated} or \cite[p. 17]{totaro2016motive}. As we have mentioned, all retracts in $\mc{C}$ are direct summands, hence $\mc{B}$ is thick if and only if it is retraction-closed. If $\mc{B}$ is thick, then it is strictly full.

\begin{lemma}\label{idem-thick}
	Let $\mc{B}$ be a strictly full triangulated subcategory of $\mc{C}$. If $\mc{B}$ is idempotent complete, then $\mc{B}$ is thick.
\end{lemma}

\begin{proof}
	Consider a decomposition $X=Y\oplus Z$ in $\mc{C}$, where $X\in\mc{B}$. Let $p:X\to Y\to X$ be the idempotent corresponding to the projection onto $Y$. Since $\mc{B}$ is idempotent complete and $p^2=p$, we have a decomposition $X=X_1\oplus X_2$ in $\mc{B}$ such that $X_1=\on{Im}p$ and $X_2=\on{Ker}p$ in $\mc{B}$. The inclusion of $\mc{B}$ in $\mc{C}$ is an additive functor, hence it preserves direct sums. It follows that $X=X_1\oplus X_2$ in $\mc{C}$, hence $X_1=\on{Im}p$ and $X_2=\on{Ker}p$ in $\mc{C}$. By the universal property of kernel and image, we deduce that $X_1\cong Y$ and $X_2\cong Z$. Since $\mc{B}$ is strictly full, we conclude that $Y,Z\in \mc{B}$, as desired.
\end{proof}

Let $\mc{C}$ be a triangulated category, and let $S$ be a collection of objects of $\mc{C}$. There exists a smallest strictly full triangulated subcategory $\mc{C}_S$ of $\mc{C}$ containing $S$. One can explicitly describe $\mc{C}_S$ as follows (see \cite[Tag 09SI]{stacks-project} for the case when $S$ consists of a single element). Let $\ang{S}_1$ be the strictly full subcategory of $\mc{C}$ consisting of objects in $\mc{C}$ isomorphic to direct summands of \emph{finite} direct sums \[\oplus_{i,j}E_i[n_j]\] where $E_i\in S$ and $n_j\in \Z$. For $n\geq 2$, define $\ang{S}_n$ inductively as the full subcategory of $\mc{C}$ whose objects are isomorphic to direct summands of objects $X$ fitting into a distinguished triangle \[A\to X\to B\to A[1],\] where $A$ is an object of $\ang{S}_1$ and $B$ is an object of $\ang{S}_{n-1}$. 

\begin{lemma}
	Let $\mc{C}$ be a triangulated category, let $S$ be a collection of objects of $\mc{C}$, and for every $n\geq 1$ define $\ang{S}_n$ as above. Then the smallest strictly full triangulated subcategory $\mc{C}_S$ of $\mc{C}$ is the increasing union of the $\ang{S}_n$.
\end{lemma}

\begin{proof}
	See \cite[Tag 0ATG]{stacks-project}. There, the term \emph{saturated} is used; see \cite[Tag 05RB]{stacks-project} for the definition. It is immediate to see that a strictly full subcategory is thick if and only if it is saturated. 
\end{proof} 

We say that a triangulated subcategory $\mc{B}$ of $\mc{C}$ is \emph{negative} if for every two objects $X,Y\in \mc{B}$ and every $i>0$ we have $\on{Hom}(X,Y[i])=0$. 

The \emph{envelope} of $S$ is defined as the smallest full subcategory $\mc{B}$ of $\mc{C}$ that is extension-closed, that is, $0\in \mc{B}$, and for every distinguished triangle \[A\to B\to C\to A[1]\] in $\mc{C}$, if $A,C\in \mc{B}$ then $B\in \mc{B}$, and thick (or equivalently closed under retracts, by \cite[Lemma 2.1.5]{neeman2001triangulated}).

We say that a collection $S$ of objects of $\mc{C}$ \emph{densely generates} a triangulated subcategory $\mc{B}$ of $\mc{C}$ if $\mc{B}$ is the smallest triangulated subcategory of $\mc{C}$ closed under retractions (that is, containing all retracts of its objects in $\mc{C}$) whose set of objects contains $S$. Equivalently, $\mc{B}$ is the envelope of $\cup_{j\in \Z} S[j]$.

\subsection{Weight structures} A $t$-structure on a triangulated category $\mc{C}$ is a pair of strictly full subcategories $\mc{C}^{\leq 0}$ and $\mc{C}^{\geq 0}$ satisfying the axioms of \cite[IV.4, Definition 2]{gelfand2013methods}. The notion of a weight structure on $\mc{C}$ is a natural counterpart to the notion of $t$-structure. It was introduced independently by Bondarko \cite{bondarko2010weight} and Pauksztello \cite{pauksztello2008compact} (who used the name co-$t$-structure). Our main references for weight structures are \cite{bondarko2010weight} and \cite{bondarko2018constructing}.

The following is \cite[Definition 1.1.1]{bondarko2010weight}.
\begin{defin}
	A pair of strictly full subcategories $\mc{C}^{w\leq 0}$ and $\mc{C}^{w\geq 0}$ of $\mc{C}$ is said to define a \emph{weight structure} $w$ on $\mc{C}$ if the conditions below are satisfied. Let $\mc{C}^{w\leq n}:=\mc{C}^{w\leq 0}[-n]$ and  $\mc{C}^{w\geq n}:=\mc{C}^{w\geq 0}[-n]$.
	\begin{enumerate}
		\item $\mc{C}^{w\leq 0}$ and $\mc{C}^{w\geq 0}$ are additive and Karoubi-closed in $\mc{C}$.
		\item $\mc{C}^{w\geq 1}\c \mc{C}^{w\geq 0}, \mc{C}^{w\leq 0}\c \mc{C}^{w\leq 1}$.
		\item If $X\in \mc{C}^{w\leq 0}$ and $Y\in \mc{C}^{w\leq 1}$ we have $\on{Hom}(X,Y)=0$.
		\item For every $X\in \mc{C}$ there exist $A\in \mc{C}^{w\leq 0}, B\in \mc{C}^{w\geq 0}$, and a distinguished triangle
		\[B[-1]\to X\to A\to B.\] 
	\end{enumerate}	
	The \emph{heart} of the weight structure $w$ is the strictly full subcategory $\mc{C}^{w\leq 0}\cap \mc{C}^{w\geq 0}$.
\end{defin}

In this paper we make use of the ``cohomological convention"  $(\mc{C}^{w\leq 0},\mc{C}^{w\geq 0})$ for weight structures, as in \cite{bondarko2010weight}. In \cite{bondarko2018constructing}, the ``homological convention" $(\mc{C}_{w\leq 0},\mc{C}_{w\geq 0})$ is used. To switch between the two conventions, it suffices to remember that $\mc{C}^{w\leq 0}=\mc{C}_{w\geq 0}$ and $\mc{C}^{w\geq 0}=\mc{C}_{w\leq 0}$; see \cite[p. 41]{bondarko2018constructing}.


\begin{prop}[Bondarko]\label{bondarko1}
	Let $\mc{B}$ be a negative additive subcategory of $\mc{C}$, and assume that the objects of $\mc{B}$ densely generate $\mc{C}$. 
	\begin{enumerate}[label=(\alph*)]
		\item There exists a (unique) weight structure $w$ on $\mc{C}$ whose heart equals $\on{Kar}_{\mc{C}}(\mc{B})$. The categories $\mc{C}^{w\leq 0}$ and $\mc{C}^{w\geq 0}$ are the envelopes of $\cup_{j\geq 0}\mc{B}[j]$ and $\cup_{j\leq 0}\mc{B}[j]$, respectively.
		\item If $\mc{B}$ is idempotent complete, then the heart of $w$ is $\mc{B}$.
	\end{enumerate}
\end{prop}

\begin{proof}
	(a) This is part of \cite[Corollary 2.1.2]{bondarko2018constructing}.
	
	(b) If $\mc{B}$ is idempotent complete, then it coincides with $\on{Kar}_{\mc{C}}(\mc{B})$. The conclusion follows from (a).
\end{proof}

We say that the weight structure $w$ on $\mc{C}$ is \emph{bounded} if for every object $X\in \mc{C}$ there exist $i,j\in \Z$ such that $X\in \mc{C}^{w\leq i}$ and $X\in \mc{C}^{w\geq j}$. 

\begin{prop}[Bondarko]\label{bondarko2}
	Let $w$ be a weight structure on $\mc{C}$. If $w$ is bounded and the heart of $w$ is idempotent complete, then $\mc{C}$ is idempotent complete.
\end{prop}

\begin{proof}
	This is \cite[Lemma 5.2.1]{bondarko2010weight}.
\end{proof}

\section{The homomorphisms $H^i_{\ell}$}\label{sec4}
\subsection{Triangulated categories of motives}\label{triang-motives}
 Let $k$ be a perfect field of exponential characteristic $p$ (that is, $p=1$ if $\on{char}k=0$, and $p=\on{char}k$ if $\on{char}k>0$), and let $R$ be a ring (commutative, with identity) such that $p$ is invertible in $R$. We follow the notation of \cite[Section 5]{totaro2016motive}. We denote Voevodsky's ``big" triangulated category of motives over $k$ with $R$ coefficients by $\on{DM}(k;R)$; see \cite[p. 2095]{totaro2016motive}. It is a tensor triangulated category, with a symmetric monoidal product $\otimes$. 

If $X$ is a smooth separated $k$-scheme of finite type, we denote by $M(X)$ the motive of $X$ and by $M^c(X)$ the compactly supported motive of $X$; see \cite[p. 2096]{totaro2016motive}. We have $M(X)\otimes M(Y)\cong M(X\times_k Y)$ and $M^c(X)\otimes M^c(Y)\cong M^c(X\times_k Y)$. The category $\on{DM}(k;R)$ contains a motive $R(j)$ for every $j\in \Z$, with the properties that $R(a)\otimes R(b)\cong R(a+b)$ for all integers $a,b$ and that $R(j)[2j]=M^c(\A_k^j)$ for all $j\geq 0$; see \cite[p. 2096]{totaro2016motive}. We have $M^c(\P^1_k)=M(\P^1_k)=R\oplus R(1)[2]$. If $M\in \on{DM}(k;R)$ and $j\in \Z$, we set $M(j):=M\otimes R(j)$. 

The category $\on{DM}_{\on{gm}}(k;R)$ of geometric motives is defined as the smallest thick triangulated subcategory of $\on{DM}(k;R)$ which contains $M(X)(a)$ for all smooth separated schemes $X$ of finite type over $k$ and all integers $a$; see \cite[p. 2099]{totaro2016motive}.

We denote by $\on{DMT}(k;R)$ the smallest triangulated subcategory of $\on{DM}(k;R)$ closed under arbitrary direct sums and containing $R(j)$ for every $j\in \Z$.
We denote by $\on{DMT}_{\on{gm}}(k;R)$ the smallest	thick triangulated subcategory of $\on{DM}_{\on{gm}}(k;R)$ containing $R(j)$ for every $j\in \Z$. Objects of $\on{DM}_{\on{gm}}(k;R)$ are called mixed Tate motives, and objects  $\on{DMT}_{\on{gm}}(k;R)$ are called ``small" mixed Tate motives.

\subsection{The construction} For every closed embedding $Z\hookrightarrow X$ with open complement $U$, there is a distinguished triangle \[M^c(Z)\to M^c(X)\to M^c(U)\to M^c(Z)[1].\] Moreover, as we have mentioned, for every two smooth $k$-varieties $X,Y$, we have $M^c(X\times Y)=M^c(X)\otimes M^c(Y)$. Since $k$ is perfect, by \Cref{smoothsuffice} the association $\set{X}\mapsto [M^c(X)]$ for every smooth $k$-variety $X$ defines a ring homomorphism \[K_0(\on{Var}_k)\to K_0^{\on{tr}}(\on{DM}_{\on{gm}}(k;R)).\]
Let $\on{Mot}(k;R)$ be the category of (pure) Chow motives over $k$ with $R$-coefficients. Since $k$ is perfect and $p$ is invertible in $R$, the Voevodsky embedding $\on{Mot}(k;R)\to \on{DM}_{\on{gm}}(k;R)$ induces an isomorphism of rings \begin{equation}\label{mot-dm-iso}K_0(\on{Mot}(k;R))\xrightarrow{\sim} K_0^{\on{tr}}(\on{DM}_{\on{gm}}(k;R));\end{equation} see \cite[Corollary 6.4.3]{bondarko2009differential} if $\on{char}k=0$ and \cite[Proposition 2.3.3]{bondarko2011motivic} if $\on{char}k>0$ (here we use that $k$ is perfect). We thus obtain a ring homomorphism \[\nu_R:K_0(\on{Var}_k)\to K_0(\on{Mot}(k;R)).\] 
Let $\mc{G}$ be the absolute Galois group of $k$, and let $\ell\neq p$ be a prime number. We define $\nu_{\ell}:=\nu_{\Z_\ell}$. We let $\on{Corr}(k;\Z_{\ell})$ be the additive category whose objects are smooth projective $k$-varieties, and whose morphisms are $\Z_{\ell}$-correspondences. We have an additive functor $F:\on{Corr}(k;\Z_{\ell})\to D^b(\on{Rep}_{\mc{G},\Z_{\ell}})$, which to a smooth projective $k$-variety associates the complex $(M,d)$, where $M^i:=H^i_{\text{\'et}}(\cl{X},\Z_{\ell})$ and $d^i=0$ for every integer $i$. By the K\"unneth formula with $\Z_{\ell}$-coefficients for smooth projective varieties over an algebraically closed field (see \cite[Theorem VI.8.21]{milne1980etale}), 
we see that $F$ respects tensor products. Since $D^b(\on{Rep}_{\mc{G},\Z_{\ell}})$ is idempotent complete and $[\Z_{\ell}(1)]$ is invertible in $D^b(\on{Rep}_{\mc{G},\Z_{\ell}})$, $F$ factors through an additive functor \[\on{Mot}(k;\Z_{\ell})\to D^b(\on{Rep}_{\mc{G},\Z_{\ell}})\] which also respects the monoidal structures. We thus obtain a ring homomorphism 
\[\rho_{\ell}:K_0(\on{Mot}(k;\Z_{\ell}))\to L_0(\on{Rep}_{\mc{G},\Z_{\ell}}).\]

\begin{prop}\label{euler-properties}
	Let $k$ be a perfect field of exponential characteristic $p$, and let $\ell$ be a prime different from $p$.
	\begin{enumerate}[label=(\alph*)]
		\item The composition \[\chi_{\ell}:=\rho_{\ell}\circ\nu_{\ell}:K_0(\on{Var}_k)\to L_0(\on{Rep}_{\mc{G},\Z_{\ell}})\] has the property that for every smooth proper $k$-variety $X$, $\chi_{\ell}(\set{X})$ is the class of the complex $(M,d)$, where $M^i=H^i_{\text{\'et}}(\cl{X},\Z_{\ell})$ and $d^i=0$ for every $i\in\Z$. 
		\item The homomorphism  $\chi_{\ell}$ induces a commutative diagram of ring homomorphisms
		\[
		\begin{tikzcd}
		K_0(\on{Stacks}_k) \arrow[r,"\Phi^{-1}_{\L}\chi_{\ell}"] \arrow[d]  & \Phi^{-1}_{[\Z_{\ell}(-1)[-2]]}L_0(\on{Rep}_{\mc{G},\Z_{\ell}}) \arrow[d] \\  
		\hat{K}_0(\on{Var}_k) \arrow[r,"\hat{\chi}_{\ell}"] &  \hat{L}_0(\on{Rep}_{\mc{G},\Z_{\ell}}).
		\end{tikzcd}		
		\]
		The homomorphisms $\Phi^{-1}_{\L}\chi_{\ell}$ and $\hat{\chi}_{\ell}$ are the unique extensions of $\chi_{\ell}$ to $K_0(\on{Stacks}_k)$ and $\hat{K}_0(\on{Var}_k)$, respectively.
		\item If $\on{char}k=0$, the homomorphisms $\chi_{\ell}, \Phi^{-1}_{\L}\chi_{\ell}$ and $\hat{\chi}_{\ell}$ coincide with those introduced by Ekedahl.
	\end{enumerate}
\end{prop}

\begin{proof}
	(a) If $X$ is smooth and proper over $k$, then $\nu_{\ell}(\set{X})$ is the class of the Chow motive $(X,\on{id}_X)$ of $\on{Mot}(k;\Z_{\ell})$.
	
	(b) We note that the image of $\Phi_{\L}$ in $\hat{K}_0(\on{Var}_k)$ consists of invertible elements, and that the same is true for the image of $\Z_{\ell}(-1)[-2]$ in $\hat{L}_0(\on{Rep}_{\mc{G},\Z_{\ell}})$. By the universal property of localization, we obtain the vertical homomorphisms in the diagram.
	Using the decomposition $\set{\P^1_k}=\L+1$, we see that $\chi_{\ell}(\L)=[\Z_{\ell}(-1)[-2]]$. By the universal property of localization, $\chi_{\ell}$ uniquely extends to a homomorphism $\Phi^{-1}_{\L}\chi_{\ell}$. 
	
	The homomorphism $\chi_{\ell}:K_0(\on{Var}_k)[\L^{-1}]\to L_0(\on{Rep}_{\Z_{\ell}[\mc{G}]})[[\Z(-1)[-2]]^{-1}]$ is continuous with respect to the filtration topologies. By the universal property of the completion, it induces a unique homomorphism $\hat{\chi}_{\ell}$ at the level of completions. This completes the construction of the commutative diagram.
	
	(c) If $X$ is a smooth projective $k$-variety, by (a) we have that $\chi_{\ell}(\set{X})$ is the class of the complex $(M,d)$, where $M^i=H^i_{\text{\'et}}(\cl{X},\Z_{\ell})$ and $d^i=0$ for every $i\in \Z$. Therefore, our definition of $\chi_{\ell}$ agrees with that of Ekedahl on the classes of smooth projective $k$-varieties; see \cite[Section 3]{ekedahl2009grothendieck}. By \Cref{bittner}, $K_0(\on{Var}_k)$ is generated by the classes of smooth projective $k$-varieties.
\end{proof}

\begin{rmk}	
	When $\on{char}k\neq 0$, we do not know whether $K_0(\on{Var}_k)$ is generated by classes of smooth proper $k$-varieties. Thus, we do not know whether $\chi_{\ell}$ is determined by its behavior on classes of smooth proper $k$-varieties.
\end{rmk}

	Let $R$ be a topological commutative ring with identity, let $\mc{G}$ be a profinite group, and let $i$ be an integer. We have a group homomorphism \[h^i_{R}:L_0(\on{Rep}_{\mc{G},R})\to K_0(\on{Rep}_{\mc{G},R})\] which sends the class of a complex $(M,d)$ to $[\on{Ker}d^i/\on{Im}d^{i+1}]$. If $(M,d)$ is concentrated in degree $<i$, this is zero. It follows that, when we endow $L_0(\on{Rep}_{\mc{G},R})$ with the filtration topology and $K_0(\on{Rep}_{\mc{G},R})$ with the discrete topology, the homomorphism $h^i_R$ is continuous, and so it uniquely induces a group homomorphism \[\hat{h}^i_R:\hat{L}_0(\on{Rep}_{\mc{G},R})\to K_0(\on{Rep}_{\mc{G},R}).\] 
	
	Let now $k$ be a perfect field of exponential characteristic $p$, let $\ell\neq p$ be a prime number, and let $\mc{G}$ be the absolute Galois group of $k$. We define $h^i_{\ell}:=h^i_{\Z_{\ell}}$ and $\hat{h}^i_{\ell}:=\hat{h}^i_{\Z_{\ell}}$. For every smooth projective $k$-variety $X$ and every $m\in\Z$, the composition
	\[\hat{H}^i_{\ell}:=\hat{h}^i_{\ell}\circ \hat{\chi}_{\ell}:\hat{K}^0(\on{Var}_k)\to K_0(\on{Rep}_{\mc{G},\Z_{\ell}})\] sends $\set{X}/\L^m$ to $[H_{\text{\'et}}^{i+2m}(\cl{X},\Z_{\ell}(m))]$. The composition \[K_0(\on{Stacks}_k)\to \hat{K}_0(\on{Var}_k)\xrightarrow{\hat{H}^i_{\ell}} K_0(\on{Rep}_{\mc{G},R})\] will be denoted by $H^i_{\ell}$.
	
	If $k$ is algebraically closed, then the absolute Galois group $\mc{G}$ of $k$ is trivial, hence $K_0(\on{Rep}_{\mc{G},R})=K_0(\on{Mod}_R)$.
	
	The following observation is implicit in \cite{ekedahl2009grothendieck} and \cite{ekedahl2009geometric}.
	
	\begin{lemma}\label{comparison}
	Assume $k=\C$. Denote by $H^i:\hat{K}_0(\on{Var}_{\C})\to K_0(\on{Mod}_{\Z})$ the continuous homomorphism determined by $\set{X}/\L^m\mapsto H^{i+2m}(X(\C),\Z)$ (singular cohomology). For every prime $\ell$, let $(-)\otimes_{\Z} \Z_{\ell}:K_0(\on{Mod}_{\Z})\to K_0(\on{Mod}_{\Z_{\ell}})$ be the map induced by $[M]\mapsto [M\otimes \Z_{\ell}]$ for every finitely generated abelian group $M$. Then 
	\[H^i(\alpha)\otimes_{\Z} \Z_{\ell}=H^i_{\ell}(\alpha)\] for every $\alpha\in \hat{K}_0(\on{Var}_{\C})$. 
	\end{lemma}

\begin{proof}
	It is enough to consider those $\alpha$ of the form $\set{X}/\L^m$, where $X$ is a smooth projective variety over $\C$, and $m\in \Z$. We have $H^i(\set{X}/\L^m)=H^{i+2m}(X(\C),\Z)$ and $H^i_{\ell}(\set{X}/\L^m)=H_{\text{\'et}}^{i+2m}(X,\Z_{\ell})$. By Artin's Comparison Theorem, for every $n$ we have a canonical isomorphism
	\[H^i_{\text{\'et}}(X,\Z/\ell^n)\xrightarrow{\sim} H^i(X(\C),\Z/\ell^n).\]
	Passing to the projective limit in $n$, we get
	\[H^i_{\text{\'et}}(X,\Z_{\ell}):=\varprojlim H^i_{\text{\'et}}(X,\Z/\ell^n)\cong \varprojlim H^i(X(\C),\Z/\ell^n).\] Recall that every smooth complex variety is homeomorphic to a finite simplicial complex. Moreover, for every bounded complex $(M,d)$ of finite abelian groups and for every integer $i$ the natural map \[H^i(M)\otimes_{\Z}\Z_{\ell}\to \varprojlim H^i(M)/\ell^n\] is an isomorphism. It follows that \[\varprojlim H^i(X(\C),\Z/\ell^n)\cong H^i(X(\C),\Z_{\ell}).\]
	Since  $\Z_{\ell}$ is flat over $\Z$, we have 	$H^i(X(\C),\Z_{\ell})\cong H^i(X(\C),\Z)\otimes_{\Z}\Z_{\ell}$. We conclude that $H^i_{\ell}(\set{X}/\L^m)=H^i(\set{X}/\L^m)\otimes_{\Z}\Z_{\ell}$, as desired.
\end{proof}

\begin{rmk}
	Instead of $\on{Rep}_{\mc{G},\Z_{\ell}}$, Ekedahl used a category $\on{Coh}_k$ of \emph{mixed} Galois representations; see \cite[Section 2]{ekedahl2009grothendieck}. We could have stated \Cref{euler-properties} with $\on{Rep}_{\mc{G},\Z_{\ell}}$ replaced by $\on{Coh}_k$, without any change in the proof. Since the mixedness property is not needed to define the $H^i_{\ell}$, we decided to use the simpler $\on{Rep}_{\mc{G},\Z_{\ell}}$ instead of $\on{Coh}_k$.
\end{rmk}
	
	In \cite{bondarko2020chow}, Bondarko pursues the study of variants of $H^i_{\ell}$ from a different point of view. Bondarko's paper does not treat stacks, and focuses on unramified cohomology of sheaves with transfers.
	
	\section{A variant of $H^i_{\ell}$}\label{simplified}
	The material of this section will not be used elsewhere in this article.
	
	The discussion of the previous section simplifies if one is only interested in cohomology with coefficients $\Z/\ell$. There is a homomorphism of rings \[\sigma_{\Z/\ell}: K_0(\on{Mot}(k;\Z/\ell))\to K_0(\on{Rep}_{\mc{G},\Z/\ell})[t],\] which sends the class of a $\Z/\ell$-motive $(X,q)$ to $\sum_j[q_*H^i(\cl{X},\Z/\ell)]t^i$. The fact that $\psi_{\Z/\ell}$ is a ring homomorphism follows from the K\"unneth formula with $\Z/\ell$ coefficients; see \cite[Theorem VI.8.1]{milne1980etale}.
	\begin{prop}\label{fell}
		Let $k$ be a perfect field of exponential characteristic $p$, and let $\ell$ be a prime different from $p$.
		\begin{enumerate}[label=(\alph*)]
			\item The composition \[\psi_{\Z/\ell}:=\sigma_{\Z/\ell}\circ\nu_{\Z/\ell}:K_0(\on{Var}_k)\to K_0(\on{Rep}_{\mc{G},\Z/{\ell}})[t]\] has the property that for every smooth proper $k$-scheme $X$, we have \[\psi_{\Z/\ell}(\set{X})=\sum_i[H^i_{\text{\'et}}(\cl{X},\Z/\ell)]t^i.\] 
			\item The homomorphism  $\psi_{\Z/\ell}$ uniquely induces a commutative diagram of ring homomorphisms
			\[
			\begin{tikzcd}
			K_0(\on{Stacks}_k) \arrow[r,"\Phi^{-1}_{\L}\psi_{\Z/\ell}"] \arrow[d]  & \Phi^{-1}_{t}K_0(\on{Rep}_{\mc{G},\Z/{\ell}})[t] \arrow[d] \\  
			\hat{K}_0(\on{Var}_k) \arrow[r,"\hat{\psi}_{\Z/\ell}"] & K_0(\on{Rep}_{\mc{G},\Z/{\ell}})[\![t]\!].
			\end{tikzcd}		
			\]
		\end{enumerate}
	\end{prop}
	Entirely analogous results hold using $H^{\bullet}(-,\Q_{\ell})$, where $\ell$ is a prime number different from $p$.
	\begin{proof}
		The proof is analogous to that of \Cref{euler-properties}. 
	\end{proof}

	We now give an application of \Cref{fell} to the computation of the motivic class $\set{B_kG}\in K_0(\on{Stacks}_k)$ of a linear algebraic $k$-group $G$, and to its relation with the stable rationality of $B_kG$.
	
	Recall that $B_kG$ is said to be stably rational over $k$ if there exist a $G$-represen\-tation $V$ over $k$ and a dense open subscheme $U\c V$ such that there exists a $G$-torsor $U\to U/G$, where $U/G$ is a stably rational variety (that is, we may find $m\geq 0$ such that $U/G\times_k \A_k^m$ is birationally equivalent to some $\A_k^n$). By the no-name lemma, this definition is independent of the choice of $V$ and $U$.

	The equality $\set{B_kG}=1$ in $K_0(\on{Stacks}_k)$ holds for a large number of finite group schemes $G$ over $k$. For example, it holds if $G=\mu_n$, or a symmetric group $S_n$; see \cite[Proposition 3.2, Theorem 4.3]{ekedahl2009geometric}. On the other hand, there are also examples of finite groups $G$ such that $\set{B_kG}\neq 1$; see \cite[Theorem 5.1, Corollary 5.8]{ekedahl2009geometric}. It is striking that in all examples of finite group schemes $G$ considered in \cite{ekedahl2009geometric}, $B_kG$ is stably rational over $k$ if and only if $\set{B_kG}=1$ in $K_0(\on{Stacks}_k)$.
	
	The equality $\set{B_kG}\set{G}=1$ in $K_0(\on{Stacks}_k)$ holds for many connected linear algebraic groups $G$. For example, it holds if $G$ is special (i.e. $H^1(K,G)=0$ for every field extension $K/k$), $G=\on{PGL}_2$ and $\on{PGL}_3$ by \cite{bergh2015motivic}, $G=\on{SO}_n$ by \cite{dhillon2016motive} and \cite{talpo2017motivic}, and $G$ a split group of type $G_2$, $\on{Spin}_7$ and $\on{Spin}_8$ by \cite{pirisi2017motivic}. In all these examples $B_kG$ is stably rational.

	Assume that $k$ is a field of characteristic zero admitting a biquadratic extension. In \cite[Theorem 1.5]{scavia2018motivic}, we constructed the first example of a connected group $G$ such that $\set{B_kG}\set{G}\neq 1$. It has the property that $B_kG$ is stably rational. In \cite[Theorem 1.6]{scavia2018motivic}, we gave an example of a finite group scheme $A$ such that $B_kA$ is stably rational but $\set{B_kA}\neq 1$.
	
	As an application of \Cref{fell}, we generalize the results of \cite{scavia2018motivic} to the case where the characteristic of the base field is different from $2$. 
	
	\begin{prop}
		Let $k$ be a perfect field of characteristic different from $2$ which admits a biquadratic field extension. 
	\begin{enumerate}[label=(\alph*)]
		\item There exists a $k$-torus $G$ of rank $3$ such that $B_kG$ is stably rational but $\set{B_kG}\set{G}\neq 1$ in $K_0(\on{Stacks}_k)$.
		\item There exists a finite \'etale group $k$-scheme of multiplicative type $A$ such that $B_kA$ is stably rational
		and $\set{B_kA}\neq 1$ in $K_0(\on{Stacks}_k)$.
	\end{enumerate}	
	\end{prop}
	
	\begin{proof}
		(a) Let $K/k$ be a biquadratic extension, let $\Gamma:=\on{Gal}(K/k)\cong (\Z/2)^2$, let $E_1,E_2,E_3$ be the intermediate quadratic extensions of $K/k$, and let $\sigma_i\in \Gamma$ be the unique element such that $K^{\sigma_i}=E_i$, for $i=1,2,3$. Let $G:=R^{(1)}_{E_1\times E_2/k}(\G_{\on{m}})$ be the $k$-torus of rank $3$ used in the proof of \cite[Theorem 1.5]{scavia2018motivic}. As in the first part of the proof of \cite[Theorem 1.5]{scavia2018motivic} (which works as long as $\on{char}k\neq 2$), we have that $B_kG$ is stably rational, and moreover $\set{B_kG}\set{G}=1$ in $K_0(\on{Stacks}_k)$ if and only if 
		\[\set{K}-\set{E_1}-\set{E_2}-\set{E_3}+2=0\] in $K_0(\on{Stacks}_k)$. Recall that the natural homomorphism $K_0(\on{Var}_k)\to K_0(\on{Stacks}_k)$ identifies $K_0(\on{Stacks}_k)$ with $\Phi^{-1}_{\L}K_0(\on{Var}_k)$. It follows that
		\[(\set{K}-\set{E_1}-\set{E_2}-\set{E_3}+2)f(\L)=0,\] for some monic polynomial $f(\L)$ with integer coefficients. Now we apply $\psi^i_{\Z/2}$ and consider the leading coefficients of the resulting polynomial in $t$:
		\begin{equation}\label{lindip}[\F_2[\Gamma]]-[\F_2[\Gamma/\ang{\sigma_1}]]-[\F_2[\Gamma/\ang{\sigma_2}]]-[\F_2[\Gamma/\ang{\sigma_3}]]+2[\F_2]=0\end{equation}	in $K_0(\on{Rep}_{\mc{G},\F_2})$. Here $\mc{G}$ acts on $\Gamma$ via the natural surjection $\mc{G}\to \Gamma$.
		It follows from the Krull-Schmidt Theorem that $K_0(\on{Rep}_{\mc{G},\F_2})$ is freely generated on the set of isomorphism classes of indecomposable $\mc{G}$-representations; see \cite[Lemma 4.1]{scavia2018motivic}. This is in contradiction with (\ref{lindip}), hence $\set{B_kG}\set{G}\neq 1$.
		
		(b) Let $A:=G[2]$ be the $2$-torsion subgroup of $G$. From the proof of \cite[Theorem 1.6]{scavia2018motivic}, we have $\set{B_kA}=\set{B_kG}\set{G}$. By (a), it follows that $\set{B_kA}\neq 1$ in $K_0(\on{Stacks}_k)$.
	\end{proof}

\section{Proof of Theorem \ref{ekedahl-inv}}\label{sec-ekedahl-inv}
Let $k$ be a perfect field of exponential characteristic $p$, and let $R$ be a commutative ring with identity such that $p$ is invertible in $R$. Recall that we denote by $\on{DMT}_{\on{gm}}(k;R)$ the smallest thick triangulated subcategory of $\on{DM}_{\on{gm}}(k;R)$ containing $R(j)$ for every $j\in \Z$. 

\begin{lemma}\label{dmt-idempotent}
	Assume that $R$ is a principal ideal domain. Then $\on{DMT}_{\on{gm}}(k;R)$ is the smallest strictly full subcategory of $\on{DM}_{\on{gm}}(k;R)$ containing $R(j)$ for every $j\in\Z$. 
\end{lemma}	

\begin{proof}
	Let $\mc{T}$ be the smallest strictly full triangulated subcategory of $\on{DM}_{\on{gm}}(k;R)$ containing $R(j)$ for every $j\in\Z$. To prove that $\mc{T}$ coincides with $\on{DMT}_{\on{gm}}(k;R)$, it suffices to show that $\mc{T}$ is thick.
	Consider the strictly full additive subcategory $\mc{T}_0$ of $\mc{T}$ consisting of finite direct sums of $R(j)[2j]=M^c(\A_k^j)$ for every $j\in \Z$. Triangulated subcategories are closed under translation, hence $R(j)$ belongs to any triangulated subcategory containing $R(j)[2j]$. This shows that $\mc{T}$ is densely generated by $\mc{T}_0$. For every choice of integers $i,j,h$, we have
	\begin{align*}
	\on{Hom}(R(i)[2i], R(j)[2j+h])\cong &\on{Hom}(R(i)[2i-h],R(j)[2j])\\ \cong&H^M_{2i-h}(R(j)[2j],R(i))\\ \cong&H_{2i-h}^M(\A_k^j,R(i))\\ \cong&\on{CH}^{j-i}(\A_k^j,-h),
	\end{align*}	
	where $H^M_*(-,R(*))$ denotes motivic homology.
	It follows immediately that if $h>0$ then \[\on{Hom}(R(i)[2i], R(j)[2j+h])=0,\] hence $\mc{T}_0$ is negative. Moreover, letting $h=0$ we obtain that if $i\neq j$ then \[\on{Hom}(R(i)[2i], R(j)[2j])=\on{CH}^{j-i}(\A_k^j)=0,\] and that $\on{End}(R(i)[2i])=R$. It follows that $\mc{T}_0$ is equivalent to the category of finitely generated graded free $R$-modules ($R(i)[2i]$ corresponds to a copy of $R$ in degree $i$). The idempotent completion of this category is the category of finitely generated graded projective $R$-modules. Since $R$ is a principal ideal domain, every finitely generated projective module is free, hence $\mc{T}_0$ is idempotent complete. 
	
	By \Cref{bondarko1} applied to $\mc{C}=\mc{T}$ and $\mc{B}=\mc{T}_0$, there exists a weight structure $w$ on $\mc{T}$ whose heart equals $\mc{T}_0$. It is clear that $w$ is bounded, hence we may apply \Cref{bondarko2}. We deduce that $\mc{T}$ is idempotent complete. By \Cref{idem-thick}, we conclude that $\mc{T}$ is thick, as desired.
\end{proof}	

Recall that we denote by $\L$ the class of $\A^1_k$ in $K_0(\on{Var}_k)$. We denote by $\L$ the class of the Lefschetz motive in $K_0(\on{Mot}(k;R))$, so that $[\P^1_k]=1+\L$ in $K_0(\on{Mot}(k;R))$, and also the class of $R(1)[2]=M^c(\A_k^1)$ in $K_0^{\on{tr}}(\on{DMT}_{\on{gm}}(k;R))$. The homomorphisms \[K_0(\on{Var}_k)\xrightarrow{\nu_R} K_0(\on{Mot}(k;R))\xrightarrow{\sim} K_0^{\on{tr}}(\on{DM}_{\on{gm}}(k;R))\] send $\L\mapsto\L\mapsto \L$; the isomorphism on the right is (\ref{mot-dm-iso}). Note that for every $i,j\in\Z$, we have $[R(i)[j]]=(-1)^j\L^i$ in $K_0^{\on{tr}}(\on{DM}_{\on{gm}}(k;R))$.

\begin{prop}\label{polynomial}
	Let $R$ be a principal ideal domain, and let $M$ be an object of $\on{DMT}_{\on{gm}}(k;R)$. Then $[M]\in K_0^{\on{tr}}(\on{DM}_{\on{gm}}(k;R))$ is a polynomial in $\L,\L^{-1}$.
\end{prop}

\begin{proof}
	Define $S:=\set{R(j):j\in \Z}$. By \Cref{dmt-idempotent}, $M\in \ang{S}_n$ for some $n\geq 1$. We prove that $[M]\in K_0^{\on{tr}}(\on{DM}_{\on{gm}}(k;R))$ is a polynomial in $\L$ and $\L^{-1}$ by induction on $n$. If $n=1$, then $M$ is a finite direct sum of $R(i)[j]$, and the conclusion is immediate. If $M\in \ang{S}_n$ for some $n>1$, there exists a distinguished triangle \[A\to M\to B\to A[1]\] such that $A\in \ang{S}_1$ and $B\in \ang{S}_{n-1}$. We have $[M]=[A]+[B]$ in $K_0^{\on{tr}}(\on{DM}_{\on{gm}}(k;R))$. By inductive assumption both $[A]$ and $[B]$ are polynomials in $\L$ and $\L^{-1}$, hence so is $[M]$. 
\end{proof}

\begin{proof}[Proof of \Cref{ekedahl-inv}]
	(a) Fix an embedding of $G$ in $\on{GL}_n$, for some $n\geq 1$. By \cite[Corollary 8.15]{totaro2016motive}, $\on{GL}_n/G$ is $\Z_{\ell}$-mixed Tate. Since $M^c(\on{GL}_n/G)$ belongs to $\on{DM}_{\on{gm}}(k;\Z_{\ell})$, by \cite[Theorem 7.2(4)]{totaro2016motive} we have that $M^c(\on{GL}_n/G)$ belongs to the smallest thick subcategory of $\on{DM}_{\on{gm}}(k;R)$ that contains $R(j)$ for all integers $j$, that is, it belongs to $\on{DMT}_{\on{gm}}(k;\Z_{\ell})$. By \Cref{polynomial}, we deduce that $[M^c(\on{GL}_n/G)]\in K_0(\on{DM}_{\on{gm}}(k;R))$ is a polynomial in $\L,\L^{-1}$. It follows that $\nu_{\ell}(\set{\on{GL}_n/G})$ is a polynomial in $\L,\L^{-1}$.
	
	By \cite[Proposition 3.1(i)]{ekedahl2009geometric}, we have $\set{B_kG}=\set{\on{GL}_n/G}\set{B_k\on{GL}_n}$ in $K_0(\on{Stacks}_k)$. By \cite[Proposition 1.1]{ekedahl2009grothendieck} we have
	\[\set{B_k\on{GL}_n}=\set{\on{GL}_n}^{-1}=\prod_{h=0}^{n-1}(\L^n-\L^h)^{-1}\] in $K_0(\on{Stacks}_k)$.
	It follows that  \[\nu_{\ell}(\set{B_kG})=\nu_{\ell}(\set{\on{GL}_n/G})\nu_{\ell}(\set{B_k\on{GL}_n})\] is a rational function of $\L$ in $\Phi_{\L}^{-1}K_0(\on{Mot}(k;\Z_{\ell}))$, that is, it can be expressed as a fraction $f(\L)/g(\L)$, where $f(\L)$ is a polynomial in $\L$ and $g(\L)$ is a product of elements of $\Phi_{\L}$.
	
	Following \cite[Section 2]{ekedahl2009grothendieck}, we let $K_0(\on{Coh}_{k,\ell})$ be the Grothendieck ring of the category of mixed Galois $\Q_{\ell}$-representations; see also \cite[\S 2.2]{bergh2015motivic}. By \cite[p. 6]{ekedahl2009approximating}, we have a ring homomorphism \[\chi_c:\Phi_{\L}^{-1}{K}_0(\on{Var}_k)\to \Phi_{[\Q_{\ell}(-1)[-2]]}^{-1}{K}_0(\on{Coh}_{k,\ell}),\] which sends the class of a smooth $k$-variety $X$ to $\sum_j(-1)^j[H^j_{\text{\'et},c}(\cl{X},\Q_{\ell})]$. We also have a ring homomorphism \[\chi_c':\Phi_{\L}^{-1}K_0(\on{Mot}(k;\Z_{\ell}))\to \Phi_{[\Q_{\ell}(-1)[-2]]}^{-1}{K}_0(\on{Coh}_{k,\ell}),\] which sends the class of a Chow motive $(X,q)$ to $\sum_j(-1)^j[q_*H^j_{\text{\'et}}(\cl{X},\Q_{\ell})]$. We have $\chi_c=\chi'_c\circ \nu_{\ell}$: by \Cref{bittner}, it suffices to check this on the classes of smooth projective varieties, where it is obvious.
	
	We denote by $\varphi: \Phi^{-1}_t\Z[t]\to \Phi^{-1}_{\L}K_0(\on{Var}_k)$ the ring homomorphism given by $t\mapsto \L$. Since $(\nu_{\ell}\circ \varphi)(t)=\L$ and $\nu_{\ell}(\set{B_kG})$ is a rational function of $\L$, we deduce that $\nu_{\ell}(\set{B_kG})$ belongs to the image of $\nu_{\ell}\circ \varphi$. The composition $\chi_c\circ \varphi$ is injective; see \cite[\S 2.2]{bergh2015motivic}. It follows that the restriction of $\chi'_c$ to the image of $\nu_{\ell}\circ \varphi$ is injective. By \cite[Proposition 3.1]{ekedahl2009geometric}, $\chi'_c(\nu_{\ell}(\set{B_kG}))=\chi_c(\set{B_kG})=1$, hence $\nu_{\ell}(\set{B_kG})=1$. It follows that $H^0_{\ell}(\set{B_kG})=[\Z_{\ell}]$ and $H^{i}_{\ell}(\set{B_kG})=0$ for every $i\neq 0$.
	
	(b) By \cite{totaro2016motive}, $B_kG$ is $\Z_{\ell}$-mixed Tate for every $\ell$. The conclusion now follows from (a) and \Cref{comparison}.
\end{proof}

\section{Proof of Theorem \ref{calculation}}\label{sec-calculation}

Recall that a smooth $k$-variety $Y$ is said to be separably unirational if there exists a separable rational map $\P^n_k\dashrightarrow Y$ for some $n\geq 1$.

\begin{lemma}\label{ladic}
	Let $k$ be an algebraically closed field of positive characteristic $p$, let $Y$ be a smooth projective $k$-variety of dimension $d$, and let $\ell$ be a prime different from $p$.
	\begin{enumerate}[label=(\alph*)]
		\item For every $i>0$, we have $H_{\textrm{\'et}}^{2d+i}(Y,\Z_{\ell})=0$.
		\item We have an isomorphism $H_{\textrm{\'et}}^{2d}(Y,\Z_{\ell}(d))\cong \Z_{\ell}$.
		\item If $Y$ is separably unirational, $H_{\textrm{\'et}}^{2d-1}(Y,\Z_{\ell}(d))=0$.
		\item If $Y$ is separably unirational, $H_{\textrm{\'et}}^{2d-2}(Y,\Z_{\ell}(d))_{\on{tors}}\cong \on{Hom}(\on{Br}(Y)\set{\ell},\Q_{\ell}/\Z_{\ell})$.
	\end{enumerate}
\end{lemma}

\begin{proof}
	(a) By \cite[Theorem VI.1.1]{milne1980etale}, for every $n\geq 1$ we have $H^{2d+i}_{\text{\'et}}(Y,\Z/\ell^n)=0$. It now suffices to pass to the inverse limit in $n$.
	
	(b) By \cite[Theorem VI.11.1(a)]{milne1980etale}, we have an isomorphism $H^{2d}_{\text{\'et}}(Y,\Z/\ell^n(d))\cong \Z/\ell^n$ for every $n\geq 1$. These isomorphisms are compatible with the surjections $\Z/\ell^{n+1}\to \Z/\ell^n$. Now pass to the inverse limit in $n$.
	
	(c) Let $\cl{y}\in Y(k)$. By a result of Koll\'ar, based on the de Jong-Graber-Harris-Starr Theorem, the algebraic fundamental group $\pi_1(Y,\cl{y})$ of $Y$ is trivial; see \cite[Corollary 3.7]{debarre2011rational}. By \cite[Remark I.5.4 and p. 126]{milne1980etale}, it follows that \[H^1_{\text{\'et}}(Y,\Z/\ell^n)\cong \on{Hom}_{\on{cont}}(\pi_1(Y,\cl{y}),\Z/\ell^n)=0.\] Using \cite[Corollary VI.11.2]{milne1980etale}, we obtain that $H^{2d-1}_{\text{\'et}}(Y,\Z/\ell^n(d))=0$. To complete the proof, it suffices to pass to the inverse limit in $n$. 
	
	(d) We have a non-degenerate pairing
	\[H^3_{\text{\'et}}(Y,\Z_{\ell})_{\on{tors}}\times H_{\text{\'et}}^{2d-2}(Y,\Z_{\ell}(d))_{\on{tors}}\to \Q_{\ell}/\Z_{\ell};\] see \cite[(8.10)]{grothendieck1968brauer3}. It induces an isomorphism \begin{equation}\label{ladic-eq}H_{\text{\'et}}^{2d-2}(Y,\Z_{\ell}(d))_{\on{tors}}\cong \on{Hom}(H^3_{\text{\'et}}(Y,\Z_{\ell})_{\on{tors}},\Q_{\ell}/\Z_\ell).\end{equation}
	We also have a short exact sequence
	\[0\to (\Q_{\ell}/\Z_{\ell})^{b_2-\rho}\to \on{Br}(Y)\set{\ell}\to H^3_{\text{\'et}}(Y,\Z_{\ell})_{\on{tors}}\to 0,\] where $b_2:=\on{dim}_{\Q_{\ell}}H^2_{\text{\'et}}(Y,\Q_{\ell})$ and $\rho$ is the rank of the N\'eron-Severi group of $Y$; see \cite[End of p. 148]{grothendieck1968brauer3} or \cite[Proposition 4.2.6(i)]{colliot2019brauer}.
	Since $Y$ is separably unirational, it is unirational, hence it is rationally chain connected. By \cite[Theorem 1.2]{gounelas2018fano}, we deduce that $b_2=\rho$. We deduce that $H^3_{\text{\'et}}(Y,\Z_{\ell})_{\on{tors}}\cong \on{Br}(Y)\set{\ell}$, and the conclusion now follows from (\ref{ladic-eq}).
\end{proof}


\begin{proof}[Proof of \Cref{calculation}]
	Let $V$ be a faithful $d$-dimensional $G$-representation over $\Q$ (for example, the regular $G$-representation). Let $U$ be the non-empty open subscheme of $V$ where $G$ acts trivially. Using resolution of singularities, we may write 
	\begin{equation}\label{calc-1}\set{{U}/G}=\set{{X}}+\sum_q n_q\set{X_q},\end{equation} where $X$ is a smooth projective $\Q$-variety of dimension $d$, and the $X_q$ are smooth projective $\Q$-varieties of dimension at most $d-1$.
	A standard spreading-out argument shows that there exists a positive integer $n$ with the following property: there exists a (trivial) vector bundle $\mc{V}\to \on{Spec}\Z[1/n]$, such that $G$ acts linearly and faithfully on $\mc{V}$ over $\on{Spec}\Z[1/n]$, so that the locus $\mc{U}$ on which $G$ acts freely is an open and fiber-wise dense subscheme of $\mc{V}$, and there are smooth projective morphisms $\mc{X}\to \on{Spec}\Z[1/n]$ of relative dimension $d$ and $\mc{X}_q\to \on{Spec}\Z[1/n]$ of relative dimension at most $d-1$ such that \[\set{\mc{U}/G}=\set{\mc{X}}+\sum_q n_q\set{\mc{X}_q}\] in $K_0(\on{Sch}_{\Z[1/n]})$, and such that we have $G$-equivariant $\Q$-isomorphisms $\mc{V}_{\Q}\cong V$, $\mc{U}_{\Q}\cong U$, and $\Q$-isomorphisms $\mc{X}_{\Q}\cong X$ and $(\mc{X}_q)_{\Q}\cong X_q$, for all $q$. Moreover, we may assume that there exists an isomorphism between open fiber-wise dense subschemes of $\mc{U}/G$ and $\mc{X}$ over $\on{Spec}\Z[1/n]$. We note that since $G$ is finite and $\mc{V}\to \on{Spec}\Z[1/n]$ is affine, the quotient $\mc{U}\to \mc{U}/G$ exists and is a $G$-torsor.
	Let $k$ be an algebraically closed field of characteristic $p>n$, so that there is a morphism $\on{Spec}k\to \on{Spec}\Z[1/n]$. If $\mc{S}$ is a scheme over $\Z[1/n]$, we define $\mc{S}_k:=\mc{S}\times_{\Z[1/n]}k$. Pulling back along $\on{Spec}k\to \on{Spec}\Z[1/n]$, we obtain a faithful $G$-representation $\mc{V}_k$, such that $G$ acts freely on the non-empty open subscheme $\mc{U}_k$, and
	\begin{equation}\label{calc0}\set{\mc{U}_k/G}=\set{\mc{X}_k}+\sum_q n_q\set{(\mc{X}_q)_k}\end{equation} in $K_0(\on{Var}_k)$. 	
	By \cite[Theorem 3.4]{ekedahl2009geometric}, we have 
	\begin{equation}\label{calc1}\set{B_{\Q}G}\L^d=\set{U/G}+\sum_j m_j\set{B_{\Q}H_j}\L^{a_j}\end{equation}
	in $K_0(\on{Stacks}_{\Q})$, and
	\begin{equation}\label{calc2}\set{B_kG}\L^d=\set{\mc{U}_k/G}+\sum m_j\set{B_kH_j}\L^{a_j}\end{equation} in $K_0(\on{Stacks}_{k})$, where the $H_j$ are distinct proper subgroups of $G$, $a_j<d$, and $m_j\in\Z$. Since $\L$ is invertible in $K_0(\on{Stacks}_k)$, we may divide (\ref{calc2}) by $\L^{d}$:
	\[\set{B_kG}=\set{\mc{U}_k/G}\L^{-d}+\sum_j m_j\set{B_kH_j}\L^{a_j-d}.\]
	We apply $H^i_{\ell}$ to both sides of the equation and use (\ref{calc0}):
	\begin{align}\label{flag}
	H^i_{\ell}(\set{B_kG})=H_{\ell}^{i+2d}(\set{\mc{U}_k/G})+\sum_j m_j H_{\ell}^{i+2(d-a_j)}(\set{B_kH_j})\\
	\nonumber=[H^{i+2d}_{\text{\'et}}(\mc{X}_{{k}},\Z_{\ell}(d))]+\sum_q n_q [H^{i+2d}_{\text{\'et}}((\mc{X}_q)_{{k}},\Z_{\ell}(d))]+\sum_j m_jH_{\ell}^{i+2(d-a_j)}(\set{B_kH_j}).
	\end{align}
	By (\ref{calc-1}) and (\ref{calc1}), we also have 
	\begin{align}\label{flag'}
	H^i_{\ell}(\set{B_{\cl{\Q}}G})=[H^{i+2d}_{\text{\'et}}(X_{\cl{\Q}},\Z_{\ell}(d))]+\sum_q n_q& [H^{i+2d}_{\text{\'et}}((X_q)_{\cl{\Q}},\Z_{\ell}(d))]+\\&+\sum_j m_jH_{\ell}^{i+2(d-a_j)}(\set{B_{\cl{\Q}}H_j}).\nonumber
	\end{align}
	If $G$ is the trivial group, then the conclusion is clear. Assume now $G$ is non-trivial, and that the conclusion holds for all $i\in\Z$ and all proper subgroups of $G$. This means that for every $j$ there exists $N_j$ such that the proposition holds for $H_j$ for every prime $p>N_j$. Let $N$ be an integer strictly greater than $n$, the $N_j$, and all primes dividing the order of $G$, and assume that $\on{char}k=p\geq N$.
	
	We start by showing that $H^i_{\ell}(\set{B_kG})$ is torsion for every $\ell\neq p$. By \cite[Theorem 5.1]{ekedahl2009geometric}, the classes $H^i_{\ell}(\set{B_{\cl{\Q}}G})$ and $H^i_{\ell}(\set{B_{\cl{\Q}}H_j})$ in $K_0(\on{Mod}_{\cl{\Q}})$ are torsion for all $i\neq 0$, and equal to $[\Z_{\ell}]$ if $i=0$. Therefore, (\ref{flag'}) implies \[\on{dim}_{\Q_{\ell}}H^i_{\text{\'et}}(X_{\cl{\Q}},\Q_{\ell}(d))+\sum_q n_q\on{dim}_{\Q_{\ell}}H^i_{\text{\'et}}((X_q)_{\cl{\Q}},\Q_{\ell}(d))+r_i=0,\] where $r_i$ is the sum of the $m_j$, over those $j$ satisfying $i+2(d-a_j)=0$.
	Let $R$ be a discrete valuation ring with residue field $k$, a field of fractions $K$ of characteristic zero, and such that $N$ is invertible in $R$. 
	By the invariance of \'etale cohomology under extensions of separably closed field  \cite[Corollary VI.2.6]{milne1980etale}, we have \[\on{dim}_{\Q_{\ell}}H^i_{\text{\'et}}(X_{\cl{K}},\Q_{\ell}(d))+\sum_q n_q\on{dim}_{\Q_{\ell}}H^i_{\text{\'et}}((X_q)_{\cl{K}},\Q_{\ell}(d))+r_i=0.\]
	By the proper base change theorem in \'etale cohomology \cite[Theorem VI.2.1]{milne1980etale}, we deduce that
	\begin{equation}\label{torsion}\on{dim}_{\Q_{\ell}}H^i_{\text{\'et}}(\mc{X}_{k},\Q_{\ell}(d))+\sum_q n_q\on{dim}_{\Q_{\ell}}H^i_{\text{\'et}}((\mc{X}_q)_{k},\Q_{\ell}(d))+r_i=0.\end{equation}
	Combining (\ref{flag}), (\ref{torsion}) and the inductive assumption, we deduce that $H^i_{\ell}(\set{B_kG})$ is torsion.
	
	Assume now that $i>0$. By the inductive assumption and \Cref{ladic}(a) every term on the right hand side of (\ref{flag}) is zero, hence $H^i_{\ell}(\set{B_kG})=0$. 
	
	If $i=0$, by \Cref{ladic}(a) and (b) the only non-zero term on the right-hand side of (\ref{flag}) is $H^{2d}_{\text{\'et}}(\mc{X}_{{k}},\Z_{\ell}(d))\cong \Z_{\ell}$. 
	
	If $i=-1$, by \Cref{ladic}(c) the first term in the right-hand side of (\ref{flag}) is \[[H^{2d-1}_{\text{\'et}}(\mc{X}_{{k}},\Z_{\ell}(d))]=0.\] Note that $\mc{X}_{{k}}$ is birationally equivalent to $\mc{U}_k/G$, which is separably unirational since $\on{char}k$ does not divide the order of $G$. By \Cref{ladic}(a) and (b), together with the inductive assumption, the other terms in the right-hand side of (\ref{flag}) are torsion-free (that is, integer multiples of $[\Z_{\ell}])$. It follows that $H^{-1}_{\ell}(\set{B_kG})$ is torsion-free. Since it is also torsion, it must be zero.
	
	Assume now that $i=-2$. By the inductive assumption we have
	\[
	H_{\ell}^{2(d-a_j-1)}(\set{B_kH_j})=
	\begin{cases}
	0\quad &\text{if $a_j<d-1$,}\\
	[\Z_{\ell}] &\text{if $a_j=d-1$.}\\
	\end{cases}
	\]
	Since $\on{dim}(\mc{X}_q)_k\leq d-1$, by \Cref{ladic}(a) and (b) $H^{2d-2}_{\text{\'et}}((\mc{X}_q)_{{k}},\Z_{\ell}(d))$ is torsion-free for every $q$. Since $H^{-2}_{\ell}(\set{B_kG})$ is torsion, it now follows from (\ref{flag}) that \[H^{-2}_{\ell}(\set{B_kG})=[H^{2d-2}(\mc{X}_k,\Z_{\ell}(d))_{\on{tors}}].\] 
	By \Cref{ladic}(d), we conclude that \[H^{-2}_{\ell}(\set{B_kG})=[\on{Hom}(\on{Br}(\mc{X}_k)\set{\ell},\Q_{\ell}/\Z_{\ell})].\] Since $\mc{X}_k$ is birationally equivalent to $\mc{V}_{k}/G$, we have $\on{Br}(\mc{X}_{{k}})\cong \on{Br}_{\on{nr}}({k}(B_kG))$. This completes the proof.
\end{proof}

If $k$ is a field of characteristic zero, there exists a finite group $G$ such that $\set{B_kG}\neq 1$ in $K_0(\on{Stacks}_k)$; see \cite[Corollary 5.2]{ekedahl2009geometric}. We may now extend this statement to sufficiently large positive characteristic.

\begin{cor}\label{saltman}
	There exist a finite group $G$ and an integer $N\geq 0$ such that for every field $k$ of positive characteristic $p\geq N$, $\set{B_kG}\neq 1$ in $K_0(\on{Stacks}_k)$.  
\end{cor}

\begin{proof}
	Let $q$ be a prime, and let $G$ be a finite $q$-group constructed in \cite[Theorem 3.5]{saltman1984noether}. It has the property that $\on{Br}_{\on{nr}}(V)^G\neq 0$ for every field $k$ of characteristic different from $q$, and every faithful $\cl{k}$-representation $V$ of $G$. Let $N\geq 0$ be the constant given by applying \Cref{calculation} to $G$. We may assume that $N>q$. Then for every field $k$ of characteristic $p\geq N$, we have $H^2_q(\set{B_kG})\neq 0$, where we view $\set{B_kG}$ as an element of $K_0(\on{Stacks}_{\cl{k}})$. We conclude that $\set{B_kG}\neq 1$ in $K_0(\on{Stacks}_{\cl{k}})$, hence $\set{B_kG}\neq 1$ in $K_0(\on{Stacks}_{{k}})$.
\end{proof}

\section*{Acknowledgements}
I thank Mikhail Bondarko for helpful conversations on the topic of \cite{bondarko2011motivic}, and my advisor Zinovy Reichstein for many useful suggestions on the exposition.


\begin{thebibliography}{10}
	\bibitem{behrend2007motivic}
	Kai Behrend and Ajneet Dhillon.
	\newblock On the motivic class of the stack of bundles. 
	\newblock {\emph Adv. Math.} 212 (2007), no. 2, 617-644. 

	\bibitem{bergh2015motivic}
	Daniel Bergh.
	\newblock Motivic classes of some classifying stacks.
	\newblock {\em Journal of the London Mathematical Society}, 93(1):219--243,
	2015.
	
	\bibitem{bittner2004universal}
	Franziska Bittner.
	\newblock The universal {E}uler characteristic for varieties of characteristic
	zero.
	\newblock {\em Compositio Mathematica}, 140(4):1011--1032, 2004.
	
	\bibitem{bondarko2009differential}
	M.~V. Bondarko.
	\newblock Differential graded motives: weight complex, weight filtrations and
	spectral sequences for realizations; {V}oevodsky versus {H}anamura.
	\newblock {\em J. Inst. Math. Jussieu}, 8(1):39--97, 2009.
	
	\bibitem{bondarko2010weight}
	Mikhail~V Bondarko.
	\newblock Weight structures vs. t-structures; weight filtrations, spectral
	sequences, and complexes (for motives and in general).
	\newblock {\em Journal of K-theory}, 6(3):387--504, 2010.
	
	\bibitem{bondarko2011motivic}
	Mikhail~V Bondarko.
	\newblock $\mathbb{Z}[1/p]$-motivic resolution of singularities.
	\newblock {\em Compositio Mathematica}, 147(5):1434--1446, 2011.
	
	\bibitem{bondarko2018constructing}
	Mikhail~V Bondarko and Vladimir~A Sosnilo.
	\newblock On constructing weight structures and extending them to idempotent
	completions.
	\newblock {\em Homology, Homotopy and Applications}, 20(1):37--57, 2018.
	
	\bibitem{bondarko2020chow}
	Mikhail~V Bondarko.
	\newblock On Chow-pure cohomology and Euler characteristics for motives and
	varieties, and their relation to unramified cohomology and Brauer groups
	\newblock In preparation, 2020.
	

	
	\bibitem{colliot2019brauer}
	Jean-Louis Colliot-Th\'{e}l\`ene and and Alexei N. Skorobogatov. 
	\newblock {\em The Brauer-Grothendieck group.} (2019).
	

	
	\bibitem{debarre2011rational}
	Olivier Debarre.
	\newblock {\em Rational curves on algebraic varieties.} (2011).
	\newblock https://www.math.ens.fr/~debarre/NotesGAEL.pdf
	
	\bibitem{dhillon2016motive}
	Ajneet Dhillon and Matthew B. Young.
	\newblock The motive of the classifying stack of the orthogonal group.
	\newblock {\em Michigan Math. J.} 65 (2016), no. 1, 189-197. 
	
	\bibitem{ekedahl2009approximating}
	Torsten Ekedahl.
	\newblock Approximating classifying spaces by smooth projective varieties.
	\newblock {\em arXiv preprint arXiv:0905.1538}, 2009.
	
	\bibitem{ekedahl2009geometric}
	Torsten Ekedahl.
	\newblock A geometric invariant of a finite group.
	\newblock {\em arXiv preprint arXiv:0903.3148}, 2009.
	
	\bibitem{ekedahl2009grothendieck}
	Torsten Ekedahl.
	\newblock The {G}rothendieck group of algebraic stacks.
	\newblock {\em arXiv preprint arXiv:0903.3143}, 2009.
	
	\bibitem{gelfand2013methods}
	Sergei~I Gelfand and Yuri~I Manin.
	\newblock {\em Methods of homological algebra}.
	\newblock Springer Science \& Business Media, 2013.
	
	\bibitem{gillet1996descent}
	Henri Gillet and Christophe Soul{\'e}.
	\newblock Descent, motives, and k-theory.
	\newblock {\em Journal Fur Die Reine Und Angewandte Mathematik}, 478:127--176,
	1996.
	
	\bibitem{gounelas2018fano}
	Frank Gounelas and Ariyan Javanpeykar.
	\newblock Invariants of Fano varieties in families.
	\newblock {\em Mosc. Math. J.} 18 (2018), no. 2, 305-319. 
	
	\bibitem{grothendieck1968brauer3}
	Alexander Grothendieck.
	\newblock Le groupe de {B}rauer. {III}. {E}xemples et compl\'{e}ments.
	\newblock In {\em Dix expos\'{e}s sur la cohomologie des sch\'{e}mas}, volume~3
	of {\em Adv. Stud. Pure Math.}, pages 88--188. North-Holland, Amsterdam,
	1968.
	
	\bibitem{illusie2014travaux}
	Luc Illusie, Yves Laszlo, and Fabrice Orgogozo, editors.
	\newblock {\em Travaux de {G}abber sur l'uniformi\-sation locale et la
		cohomologie \'{e}tale des sch\'{e}mas quasi-excellents}.
	\newblock Soci\'{e}t\'{e} Math\'{e}matique de France, Paris, 2014.
	\newblock S\'{e}minaire \`a l'\'{E}cole Polytechnique 2006--2008. [Seminar of
	the Polytechnic School 2006--2008], With the collaboration of
	Fr\'{e}d\'{e}ric D\'{e}glise, Alban Moreau, Vincent Pilloni, Michel Raynaud,
	Jo\"{e}l Riou, Beno\^{\i}t Stroh, Michael Temkin and Weizhe Zheng,
	Ast\'{e}risque No. 363-364 (2014) (2014).
	
	\bibitem{joyce2007motivic}
	Dominic Joyce
	\newblock Motivic invariants of Artin stacks and `stack functions'. 
	\newblock {\em Q. J. Math.} 58 (2007), no. 3, 345-392. 
	
	\bibitem{martino2016ekedahl}
	Ivan Martino.
	\newblock The {E}kedahl invariants for finite groups.
	\newblock {\em Journal of Pure and Applied Algebra}, 220(4):1294--1309, 2016.
	
	\bibitem{martino2017introduction}
	Ivan Martino.
	\newblock Introduction to the {E}kedahl {I}nvariants.
	\newblock {\em Mathematica scandinavica}, 120(2):211--224, 2017.
	
	\bibitem{milne1980etale}
	James~S. Milne.
	\newblock {\em \'{E}tale cohomology}, volume~33 of {\em Princeton Mathematical
		Series}.
	\newblock Princeton University Press, Princeton, N.J., 1980.
	

	
	\bibitem{neeman2001triangulated}
	Amnon Neeman.
	\newblock {\em Triangulated categories}.
	\newblock Princeton University Press, 2001.
	
	\bibitem{padurariu2015groups}
	Tudor P{\u{a}}durariu.
	\newblock Groups of order $p^3$ are mixed {T}ate.
	\newblock {\em arXiv preprint arXiv:1503.04235}, 2015.
	
	\bibitem{pauksztello2008compact}
	David Pauksztello.
	\newblock Compact corigid objects in triangulated categories and
	co-t-structures.
	\newblock {\em Open Mathematics}, 6(1):25--42, 2008.
	
	\bibitem{pirisi2017motivic}
	Roberto Pirisi and Mattia Talpo.
	\newblock On the motivic class of the classifying stack of ${G}_2$ and the spin groups.
	\newblock {\em Int. Math. Res. Not. IMRN} 2019, no. 10, 3265-3298. 	
	
	\bibitem{saltman1984noether}
	David~J. Saltman.
	\newblock Noether's problem over an algebraically closed field.
	\newblock {\em Invent. Math.}, 77(1):71--84, 1984.
	
	\bibitem{scavia2018motivic}
	Federico Scavia.
	\newblock On the motivic class of an algebraic group.
	\newblock {\em Accepted by Algebra \& Number Theory. arXiv:1808.00056
		[math.AG]}, 2020.
	
	\bibitem{stacks-project}
	The {Stacks Project Authors}.
	\newblock \textit{Stacks Project}.
	\newblock http://stacks.math.columbia.edu, 2018.
	
	\bibitem{talpo2017motivic}
	Mattia Talpo and Angelo Vistoli.
	\newblock The motivic class of the classifying stack of the special orthogonal group.
	\newblock {\em Bulletin of the London Mathematical Society}, 49(5):818--823, 2017
	
	\bibitem{toen2005grothendieck}
	Bertrand To\"en. 
	\newblock Grothendieck rings of Artin n-stacks.
	\newblock arXiv preprint math/0509098 (2005).
	
	\bibitem{totaro2016motive}
	Burt Totaro.
	\newblock The motive of a classifying space.
	\newblock {\em Geom. Topol.}, 20(4):2079--2133, 2016.
	
\end{thebibliography}
\end{document}